\documentclass[en, amspaper, secnum, nobib]{./cls/mymcls}
\usepackage{stmaryrd}
\title{The primitive spectrum of a semigroup of Markov operators}

\begin{document}
\maketitle
\begin{abstract}
For a semigroup $\EuScript{S}$ of Markov operators on a space of continuous functions, we use $\EuScript{S}$-invariant ideals to describe qualitative properties of $\EuScript{S}$ such as mean ergodicity and the structure of its fixed space. For this purpose we focus on \emph{primitive $\EuScript{S}$-ideals} and endow the space of those ideals with an appropriate topology. This approach is inspired by the representation theory of C*-algebras and can be adapted to our dynamical setting.\\
In the particularly important case of Koopman semigroups, we characterize the centers of attraction of the underlying dynamical system in terms of the invariant ideal structure of $\EuScript{S}$.
\end{abstract}

\section{Introduction}
\label{sec:1}
The \emph{primitive spectrum} is a useful tool in the study of C*-algebras (see, e.g., Chapter IV of \cite{Dixm1977}, Section 4.3 of \cite{Pede1979} or Section II.6.5 of \cite{Blac2006}) and plays a crucial role in representation theory (cf. \cite{Hofm2011}). Given a C*-algebra $A$ it is defined as
\begin{align*}
\mathrm{Prim}(A)\defeq \{\ker \pi \mid 0\neq \pi \textrm{ irreducible representation of A}\}.
\end{align*}
Equipped with the \emph{hull-kernel topology} (also called \emph{Jacobson topology}) it becomes a quasi-compact $\mathrm{T}_0$-space. A nice application is the so called Dauns-Hofmann Theorem asserting that---\color{black}in the unital case---\color{black}the center of $A$ is canonically isomorphic to $\mathrm{C}(\mathrm{Prim}(A))$.\\
In this note we study a dynamical version of the primitive spectrum in the commutative and unital case. Starting from a right amenable semigroup $\EuScript{S}$ of Markov operators on the space of continuous functions $\mathrm{C}(K)$ on some compact space $K$ we introduce the primitive spectrum $\mathrm{Prim}(\EuScript{S})$ of $\EuScript{S}$ as the set of absolute kernels of ergodic measures. Again we equip the primitive spectrum with a hull-kernel topology and obtain a quasi-compact $\mathrm{T}_0$-space. We then describe the space $\mathrm{C}(\mathrm{Prim}(\EuScript{S}))$ and give applications to topological dynamics and ergodic theory.\medskip\\
We now give a more detailed description of the results.\\
Based on two papers of H. H. Schaefer (see \cite{Scha1967} and \cite{Scha1968}) as well as Paragraph III.8 of \cite{Scha1974} we consider $\EuScript{S}$-invariant ideals and measures in  Section 2 \color{black} recalling some basic definitions and facts.\\
In the subsequent sections we introduce and study radical $\EuScript{S}$-ideals.  In Section 3 we give the definition and prove an equivalent characterization in the metric case (see Proposition \ref{radid1}). In the fourth section we then establish \color{black} a close connection between radical $\EuScript{S}$-ideals, centers of attraction appearing in topological dynamics and stability conditions of the semigroup (see Theorem \ref{maintheorem} and Theorem \ref{center}).\\
The last three sections are devoted to the primitive spectrum of $\EuScript{S}$ as a topological space  and its applications. In Section 5 we define the topology (cf. Proposition \ref{closureop}), state its basic properties and give some examples.\\
In Section 6 we then prove a Dauns-Hofmann-type theorem showing that if $\EuScript{S}$ is radical free (see Definition \ref{hullker}) the space of continuous functions on the primitive spectrum $\mathrm{C}(\mathrm{Prim}(\EuScript{S}))$ is canonically isomorphic to the fixed space $\mathrm{fix}(\EuScript{S})$ of the semigroup $\EuScript{S}$ (see Theorem \ref{contfunct1}). We then extend this result to the general case of not necessarily radical free $\EuScript{S}$ (see Theorem \ref{contfunct}).\\
As an application we obtain in Section 7 a new description of mean ergodicity of semigroups of Markov operators (see Theorem \ref{meanergodiccor2}) which generalizes Schaefer's Theorem 2 of \cite{Scha1967} in two different ways. On one hand we consider the more general setting of right amenable semigroups instead of single operators. But more importantly, we obtain---in contrast to Schaefer's work---a full characterization of mean ergodicity . \color{black} We finally look at some examples illustrating these results (cf. Examples \ref{finalex}).\\
It should be pointed out that while maximal invariant ideals (which have been the central objects in \cite{Scha1967} and \cite{Scha1968}) are enough to describe mean ergodic Markov operators and semigroups, the results of our paper show that primitive ideals are the natural algebraic \color{black} structure \color{black} to describe dynamical properties of general Markov semigroups, see also Remark \ref{meanprim2} below.\medskip\\
\color{black} In the following we always assume $K$ to be a compact (Hausdorff) space. We denote the Banach lattice of continuous complex-valued functions on $K$ by $\mathrm{C}(K)$ and identify \color{black} the dual space $\mathrm{C}(K)'$ of $\mathrm{C}(K)$ \color{black} with the Banach lattice of complex regular Borel masures on $K$. Moreover, we refer to \cite{Scha1974} and \cite{Meye1991} for Banach lattices and their ideal structure  and remind the reader that the closed lattice ideals of $\mathrm{C}(K)$ coincide with the closed algebra ideals and are precisely the sets
\begin{align*}
I_L \defeq \{f \in \mathrm{C}(K)\mid f|_L = 0\}
\end{align*}
with $L \subset K$ closed. Recall also \color{black} that a positive operator $T \in \mathscr{L}(\mathrm{C}(K))$ is called \emph{Markov} if $T\mathbbm{1}=\mathbbm{1}$.\\
We now fix a semigroup $\EuScript{S} \subset \mathscr{L}(\mathrm{C}(K))$ of Markov operators which is \emph{right amenable} (cf. Section 2.3 of \cite{BeJuMi1989}) if endowed with the strong operator topology, i.e., there is a positive element $m \in \mathrm{C}_{\mathrm{b}}(\EuScript{S})'$ (called  \emph{right invariant mean}) such that $m(\mathbbm{1}) =1$ and $m(R_Sf) = m(f)$ for every $f \in \mathrm{C}_{\mathrm{b}}(\EuScript{S})$ where $R_S(f)(T) \defeq f(TS)$ for all $T, S \in \EuScript{S}$. All abelian topological semigroups and compact topological groups are amenable and, in particular, right amenable. For more examples and counterexamples we refer to \cite{Day1961} and Chapter 1 of \cite{Pate1988}.\\
Note that the important cases of semigroups generated by a single operator and one-parameter semigroups are contained in our results \color{black} since these are always abelian\color{black}.\\
Many examples of Markov operators and semigroups arise from topological dynamical systems on $K$. In fact, if $\varphi\colon K \longrightarrow K$ is a continuous mapping, then the asscoiated \emph{Koopman operator} $T_\varphi \in \mathscr{L}(\mathrm{C}(K))$ defined by $f \defeq f \circ \varphi$ for $f \in \mathrm{C}(K)$ is a Markov lattice operator. We write $\EuScript{S}_\varphi$ for the semigroup $\{T_\varphi^n\mid n \in \N_{\color{black}0\color{black}}\}$.
\section{Ergodic Measures and Primitive Ideals}
\label{sec:2}
In this section we introduce primitive $\EuScript{S}$-ideals adapting concepts from the theory of C*-algebras and start with the following definition going back to H. H. Schaefer (see \cite{Scha1967}). \color{black} Recall that an ideal $I$ of $\mathrm{C}(K)$ is called \emph{proper} if $I \neq \mathrm{C}(K)$.\color{black}
\begin{definition}
A closed proper ideal $I \subset \mathrm{C}(K)$ is an \emph{$\EuScript{S}$-ideal} if it is $\EuScript{S}$-invariant, i.e., $SI\subset I$ for each $S \in \EuScript{S}$. It is called \emph{maximal} if it is maximal among all $\EuScript{S}$-ideals with respect to inclusion.
\end{definition}
\begin{remark}
 If $I$ is an $\EuScript{S}$-ideal, then a standard application of Zorn's lemma shows that $I$ is contained in a maximal proper $\EuScript{S}$-invariant ideal. Since there are no dense \color{black} proper \color{black} ideals in $\mathrm{C}(K)$, this ideal is already closed and therefore each $\EuScript{S}$-ideal is contained in a maximal $\EuScript{S}$-ideal (cf. Proposition 1 in \cite{Scha1967}).\color{black}
\end{remark}
In \cite{Sine1968} R. Sine used the concept of a self-supporting set of a Markov operator. Generalizing this to our setting, a non-\color{black}empty closed set $L \subset K$ is called \emph{self-supporting} if the measure $S'\delta_x \in \mathrm{C}(K)'$ has support in $L$ for each $x \in L$ and $S \in \EuScript{S}$. \color{black} Recall that here the \emph{support} $\supp \mu$ of a probability measure $\mu \in \mathrm{C}(K)'$ is the smallest closed subset $A \subset K$ with $\mu(A) = 1$.\color{black}\\
Each self-supporting set $L$ defines an $\EuScript{S}$-ideal
\begin{align*}
I_L \defeq \{f \in \mathrm{C}(K)\mid f|_L = 0\}.
\end{align*}
Conversely, each $\EuScript{S}$-ideal is an $I_L$ for some self-supporting set $L$ and the mapping $L \mapsto I_L$ is bijective. Moreover, each maximal $\EuScript{S}$-ideal corresponds to a minimal self-supporting set.\\
Given an $\EuScript{S}$-ideal $I$ we call the unique self-supporting set $L$ with $I_L = I$ the \emph{support of $I$} and write $L = \supp I$.
\begin{remark}
For each $\EuScript{S}$-ideal $I$, the semigroup $\EuScript{S}$ induces a semigroup $\EuScript{S}_I$ of Markov operators on $\mathrm{C}(\supp I)$ given by 
\begin{align*}
\EuScript{S}_I \defeq \{S_I \mid S \in \EuScript{S}\}
\end{align*}
with $S_If \defeq SF|_{\supp I\color{black}}$ for $S \in \EuScript{S}$ and $f \in \mathrm{C}(\supp I\color{black})$ where $F \in \mathrm{C}(K)$ is any extension of $f$ to $K$. It is readily checked that $I$ is maximal if and only if $\EuScript{S}_I$ is irreducible, i.e., there are no non-trivial $\EuScript{S}_I$-ideals (see the corollary to Proposition III.8.2 in \cite{Scha1974}).
\end{remark}
We are primarily interested in $\EuScript{S}$-ideals defined by measures. The \emph{absolute kernel} of a measure $0 \leq \mu \in \mathrm{C}(K)'$ is
\begin{align*}
I_{\mu} \defeq \{f \in \mathrm{C}(K)\mid \langle |f|,\mu \rangle = 0\}.
\end{align*}
If $\mu$ is invariant, i.e., $S'\mu = \mu$ for each $S \in \EuScript{S}$, this is an $\EuScript{S}$-ideal.\\
We \color{black} write $\mathrm{P}_{\EuScript{S}}(K) \subset \mathrm{C}(K)'$ for the space of invariant probability measures on $K$ equipped with the weak* topology. By right amenability of $\EuScript{S}$ this is always a nonempty compact convex set (this is a simple consequence of Day's fixed point theorem, see  Theorem 3 of \cite{Day1961}\color{black}).\\
\color{black} We recall that for each $\mu \in \mathrm{P}_{\EuScript{S}}(K)$ the space $\mathrm{L}^1(K,\mu)$ is the completion of $\mathrm{C}(K)/I_{\mu}$ with respect to the $\mathrm{L}^1$-norm. Since $I_{\mu}$ is $\EuScript{S}$-invariant, every $S \in \EuScript{S}$ induces an operator on $\mathrm{C}(K)/I_{\mu}$ which then uniquely extends to a bi-Markov operator $S_\mu$ on $\mathrm{L}^1(K,\mu)$, i.e., $S_\mu$ is a positive operator on $\mathrm{L}^1(K,\mu)$ with $S_\mu\mathbbm{1} = \mathbbm{1}$ and $S_\mu'\mathbbm{1} = \mathbbm{1}$. We write $\EuScript{S}_\mu \defeq \{S_\mu\mid S \in \EuScript{S}\}$ for the semigroup on $\mathrm{L}^1(K,\mu)$ induced by $\EuScript{S}$.\color{black}
\begin{definition}
A measure $\mu \in \mathrm{P}_{\EuScript{S}}(K)$ is called \emph{ergodic} if the fixed space $\fix(\EuScript{S}_{\mu})$ in $\mathrm{L}^1(K,\mu)$ is one-dimensional.
\end{definition}
The following characterization of ergodicity generalizes a result of M. Rosenblatt (cf. \cite{Rose1976}) and is well-known for single operators. We give a short proof in case of semigroup actions  inspired by the proof of Proposition 10.4 of \cite{EFHN2015}. Here and in the following we write $\mathrm{ex}\, M$ for the set of extreme points of a convex subset $M$ of a vector space.\color{black}
\begin{proposition}\label{ergodicity}
A measure $\mu \in \mathrm{P}_{\EuScript{S}}(K)$ is ergodic if and only if $\mu \in \mathrm{ex}\,\mathrm{P}_{\EuScript{S}}(K)$.
\end{proposition}
%
\color{black}
\begin{proof}
\color{black} Assume that $\fix(\EuScript{S}_\mu)$ is not one-dimensional. Since $\fix(\EuScript{S}_\mu)$ is an AL-sublattice of $\mathrm{L}^1(K,\mu)$ with weak order unit $\mathbbm{1}$, the set
\begin{align*}
B = \{f \in \fix(\EuScript{S}_\mu)\mid f \geq 0 \textrm{ and } \sup(f,\mathbbm{1}-f) = 0\}
\end{align*}
is total in $\fix(\EuScript{S}_\mu)$ (cf. page 115 of \cite{Scha1974}). But $B$ is just the set of characteristic functions in $\fix(\EuScript{S}_\mu)$. Thus there is a measurable set $A \subset K$ with $S_\mu\mathbbm{1}_{A}= \mathbbm{1}_A$ and $0 < \mu(A)<1$.
 Now consider the measures $\mu_1,\mu_2$ defined by
\begin{align*}
\mu_1(g) \defeq \frac{1}{\mu(A)}\int_A g \, \mathrm{d}\mu \textrm{ and } \mu_2(g) \defeq \frac{1}{1-\mu(A)}\int_{K\setminus A} g \, \mathrm{d}\mu
\end{align*}
for $g \in \mathrm{C}(K)$. \color{black} For every $g \in \mathrm{C}(K)$ with $0 \leq g \leq \mathbbm{1}$ and each $S \in \EuScript{S}$ we obtain
\begin{align*}
\int_A g \,\mathrm{d}\mu = \int g \wedge \mathbbm{1}_A \,\mathrm{d}\mu = \int S_\mu(g \wedge \mathbbm{1}_A) \,\mathrm{d}\mu \leq \int Sg \wedge \mathbbm{1}_A \,\mathrm{d}\mu = \int_A Sg \,\mathrm{d}\mu
\end{align*}
and, similarly \color{black}
\begin{align*}
\int_{A^c} g \,\mathrm{d}\mu \leq \int_{A^c} Sg \,\mathrm{d}\mu,
\end{align*}
which implies $\mu_i \in \mathrm{P}_\EuScript{S}(K)$ for $i=1,2$.
Moreover, 
\begin{align*}
\mu= \mu(A)\mu_1 + (1-\mu(A))\mu_2,
\end{align*}
so $\mu \notin \mathrm{ex}\,\mathrm{P}_{\EuScript{S}}(K)$.\smallskip\\
Conversely, take an ergodic measure $\mu \in \mathrm{P}_{\EuScript{S}}(K)$ and suppose that $\mu = \frac{1}{2}(\mu_1 + \mu_2)$ for some $\mu_1,\mu_2 \in \mathrm{P}_{\EuScript{S}}(K)$. Since
\begin{align*}
|\langle f,\mu_1\rangle | \leq 2\langle |f| , \mu \rangle \leq 2 \|f\|_{\mathrm{L}^1(K,\mu)}
\end{align*}
 for each $f \in \mathrm{C}(K)$ and $\mathrm{C}(K)$ is dense in $\mathrm{L}^1(K,\mu)$, we conclude that $\mu_1$ extends uniquely to a continuous functional $\tilde{\mu}_1 \in \mathrm{L}^{\infty}(K,\mu) = \mathrm{L}^1(K,\mu)'$. The semigroup $\EuScript{S}_\mu$ is mean ergodic (in the sense of Definition 8.31 of \cite{EFHN2015}) on $\mathrm{L}^1(K,\mu)$ (see Example 13.24 of \cite{EFHN2015}) and therefore $\fix(\EuScript{S}_{\mu})$ separates $\fix(\EuScript{S}_{\mu}')$ by Theorem 8.33 of \cite{EFHN2015}. Since $\fix(\EuScript{S}_{\mu})$ is one-dimensional, $\fix(\EuScript{S}_{\mu}')$ is also one-dimensional. \color{black} Consequently we obtain $\tilde{\mu}_1 = \mathbbm{1} \in \mathrm{L}^{\infty}(K,\mu)$ which implies $\mu_1 = \mu$. 
\end{proof}
\color{black} We are now ready to introduce primitive $\EuScript{S}$-ideals.
\begin{definition}
An $\EuScript{S}$-ideal $p$ is called \emph{primitive} if there is \color{black} an ergodic measure $\mu \in \mathrm{P}_{\EuScript{S}}(K)$ \color{black} with $p = I_{\mu}$.\\
The set of all primitive $\EuScript{S}$-ideals is called the \emph{primitive spectrum of $\EuScript{S}$}, denoted by $\mathrm{Prim}(\EuScript{S})$.
\end{definition}
\begin{remark}
The supports of primitive $\EuScript{S}$-ideals are precisely the supports of ergodic measures. Instead of looking at the ideal space it is therefore justified (and sometimes helpful) to see the primitive spectrum as a subset of the power set of $K$.
\end{remark}
\begin{remark}
If $\EuScript{S}\subset \mathscr{L}(\mathrm{C}(K))$ is irreducible, i.e., there are no non-trivial $\EuScript{S}$-ideals, then $\mathrm{Prim}(\EuScript{S})$ is a singleton. Other examples are given below (cf. Examples \ref{PrimExamples}).
\end{remark}
We need the following result which relates invariant measures for quotient systems to invariant measures on $K$.
\begin{proposition}\label{subsystem}
Let $L \subset K$ be the support of an $\EuScript{S}$-ideal and consider the semigroup $\EuScript{S}_{L}$ of Markov operators on $\mathrm{C}(L)$ induced by $\EuScript{S}$. The canonical continuous embedding 
\begin{align*}
i \colon \mathrm{C}(L)' \longrightarrow \mathrm{C}(K)'
\end{align*}
with $i(\mu)(f)\defeq \langle f|_L,\mu \rangle$ for each $f \in \mathrm{C}(K)$  and $\mu \in \mathrm{C}(L)'$ \color{black} restricts to continuous embeddings
\begin{align*}
&i \colon \mathrm{P}_{\EuScript{S}_L}(L) \longrightarrow \mathrm{P}_{\EuScript{S}}(K),\\
&i \colon \mathrm{ex}\,\mathrm{P}_{\EuScript{S}_L}(L)  \longrightarrow \mathrm{ex}\,\mathrm{P}_{\EuScript{S}}(K) 
\end{align*}
with
\begin{align*}
&i(\mathrm{P}_{\EuScript{S}_L}(L)) = \{\tilde{\mu} \in \mathrm{P}_{\EuScript{S}}(K) \mid \supp \tilde{\mu} \subset L\} \textrm{ and }\\
&i(\mathrm{ex}\,\mathrm{P}_{\EuScript{S}_L}(L)) = \{\tilde{\mu} \in \mathrm{ex}\,\mathrm{P}_{\EuScript{S}}(K) \mid \supp \tilde{\mu} \subset L\}.\color{black}
\end{align*}
\end{proposition}
\begin{proof}
It is obvious that images of invariant measures remain invariant. Now assume that $\mu \in \mathrm{ex}\, \mathrm{P}_\EuScript{S}(L)$ and suppose that $i(\mu) = \frac{1}{2}(\tilde{\mu}_1 + \tilde{\mu}_2)$ for measures $\tilde{\mu}_1,\tilde{\mu}_2 \in \mathrm{P}_{\EuScript{S}}(K)$. Then $\supp \tilde{\mu}_i \subset \supp \mu$ for $i=1,2$ and therefore $\tilde{\mu}_1$ and $\tilde{\mu}_2$ restrict to measures $\mu_1, \mu_2 \in \mathrm{P}_{\EuScript{S}_L}(L)$ with $\mu = \frac{1}{2}(\mu_1 + \mu_2)$, so $\mu_1 = \mu = \mu_2$ since $\mu$ is ergodic. 
\end{proof}
\begin{corollary}\label{maxprim}
Each maximal $\EuScript{S}$-ideal is primitive.
\end{corollary}
\begin{proof}
Take a maximal $\EuScript{S}$-ideal $I =I_L$. Then the induced semigroup $\EuScript{S}_I$ on $\mathrm{C}(L)$ is irreducible and consequently every $\mathrm{S}_I$-invariant measure $\mu$ is strictly positive, i.e., $\supp \mu = L$.
 
\end{proof}
We give two simple examples showing that the converse of \cref{maxprim} does not hold.
\begin{example}\label{primax}
\begin{enumerate}[(i)]
\item If $K = \T \defeq \{z \in \C\mid |z| = 1\}$ and $\varphi(z) \defeq z^2$ for $z \in \T$, then the Haar measure of $\T$ is ergodic by Proposition 2.17 of \cite{EiWa2011}. However, $1 \in \T$ is a fixed point of $\varphi$ and therefore $I_{\{1\}}$ is a non-trivial $\EuScript{S}_\varphi$-ideal. Therefore, the zero ideal is a primitive, but not maximal $\EuScript{S}_\varphi$-ideal.
\item Consider $K = \{0,1\}^{\N}$ and $\varphi((x_n)_{n \in \N}) \defeq (x_{n+1})_{n \in \N}$ for $(x_n)_{n \in \N} \in K$. Clearly, $\varphi$ has fixed points whence the zero ideal is not a maximal $\EuScript{S}_\varphi$-ideal.\\
Let $\nu \defeq \frac{1}{2}(\delta_0 + \delta_1) \in \mathrm{C}(\{0,1\})'$. Then the product measure $\mu \defeq \prod_{n \in \N} \nu \in \mathrm{C}(K)'$ on $K$ is ergodic by Proposition 6.20 of \cite{EFHN2015} and has full support. Therefore $I_{\mu} = \{0\}$ is a primitive $\EuScript{S}_\varphi$-ideal.
\end{enumerate}
\end{example}
\begin{remark}
In view of Examples \ref{primax} considering all primitive ideals instead of maximal ideals yields more information on the semigroup action.
\end{remark}
\color{black}
\section{Radical Ideals}
\label{sec:3}
The Jacobson topology on the primitive spectrum of C*-algebras can be defined using the notions of \emph{hull} and \emph{kernel} (see Section 4.3 of \cite{Pede1979} or Section II.6.5 of \cite{Blac2006}). In our context they also yield a natural correspondence between closed subsets of $\mathrm{Prim}(\EuScript{S})$ and so-called \emph{radical $\EuScript{S}$-ideals}.
\begin{definition}\label{hullker}
For subsets $A \subset \mathrm{Prim}(\EuScript{S})$ and $I \subset \mathrm{C}(K)$ we set
\begin{align*}
\ker(A)&\defeq\bigcap_{p \in A} p,\\
\mathrm{hull}(I)&\defeq \left\{p \in \mathrm{Prim}(\EuScript{S})\mid I \subset p\right\}.
\end{align*}
\begin{enumerate}[(i)]
\item For a subset $I \subset \mathrm{C}(K)$ the \emph{$\EuScript{S}$-radical} of $I$ is 
\begin{align*}
\mathrm{rad}_{\EuScript{S}}(I):= \mathrm{ker}(\mathrm{hull}(I)) )=  \bigcap_{\substack{p \in \mathrm{Prim}(\EuScript{S})\\I\subset p}}p.
\end{align*}
\item An $\EuScript{S}$-ideal $I$ is a \emph{radical $\EuScript{S}$-ideal} if $I = \mathrm{rad}_{\EuScript{S}}(I)$.
\item The semigroup $\EuScript{S}$ is \emph{radical free} if the zero ideal is a radical $\EuScript{S}$-ideal, i.e., if $\mathrm{rad}_{\EuScript{S}}(0)=0$.
\end{enumerate}
\end{definition}
We denote the set of all radical $\EuScript{S}$-ideals by $\mathrm{Rad}(\EuScript{S})$.
\begin{remark}
\color{black} We point out that our definition of a radical free semigroup does not coincide with the one of Schaefer (using maximal $\EuScript{S}$-ideals, see \cite{Scha1968}). Every radical free semigroup in the sense of Schaefer is also radical free in our terminology. However, the converse does not hold (see Examples \ref{primax}).\color{black}
\end{remark}
\begin{remark}\label{capinv}
\color{black} By the Krein-Milman theorem $\mathrm{P}_{\EuScript{S}}(K)$ is the closed convex hull of $\mathrm{ex}\,\mathrm{P}_{\EuScript{S}}(K)$ with respect to the weak* topology and therefore
\begin{align*}
\mathrm{rad}_{\EuScript{S}}(0) = \bigcap_{\mu \in \mathrm{P}_{\EuScript{S}}(K)} I_\mu.
\end{align*}	
\color{black} 
\end{remark}
\begin{remark}
The $\EuScript{S}$-radical of a subset $I \subset \mathrm{C}(K)$ is either $\mathrm{C}(K)$ or \color{black} a radical $\EuScript{S}$-ideal\color{black}. Moreover, we always have $\mathrm{hull}(I) = \mathrm{hull}(\mathrm{rad}_{\EuScript{S}}(I))$.
\end{remark}
\begin{remark}

Just as primitive ideals correspond to the supports of ergodic measures, radical ideals correspond to the closures of unions of supports of ergodic measures. Therefore \color{black} $\EuScript{S}$ is radical free if and only if the union of all supports of invariant ergodic measures is dense in $K$. Note that the latter set is \color{black} not closed in general \color{black} (see Example \ref{PrimExamples} (iii) below). For the Markov semigroup induced by the right shift on $K = \beta \N\setminus \N$ this set is nowhere dense (cf. Corollary 1.5 in \cite{Chou1967}).\\
%
\end{remark}
We need the following \color{black} result \color{black} which relates radical and primitive ideals of quotient systems to the corresponding $\EuScript{S}$-ideals of $\mathrm{C}(K)$.\color{black} 
\begin{proposition}\label{lemmaquotient}
\color{black}Let $I =I_L\subset \mathrm{C}(K)$ be an $\EuScript{S}$-ideal and $\EuScript{S}_I$ the semigroup of Markov operators on $\mathrm{C}(L)$ induced by $\EuScript{S}$. Then the mappings 
\begin{align*}
&\{p \in \mathrm{Prim}(\EuScript{S})\mid I \subseteq p\} \longrightarrow \mathrm{Prim}(\EuScript{S}_I),\quad 
p \mapsto p|_L,\\
&\{J \in \mathrm{Rad}(\EuScript{S})\mid I \subseteq J\} \longrightarrow \mathrm{Rad}(\EuScript{S}_I),\quad 
J \mapsto J|_L,
\end{align*}
where $J|_L \defeq \{f|_L \mid f \in J\}$ for $J \subset \mathrm{C}(K)$, are inclusion preserving bijections  with inclusion preserving inverses\color{black}. Moreover, \color{black} $\mathrm{rad}_{\EuScript{S}_I}(0) = \mathrm{rad}_{\EuScript{S}}(I)|_L$.\color{black}
\end{proposition}
\begin{proof}
\color{black} We first recall that the \color{black} natural projection $P\colon \mathrm{C}(K)\longrightarrow \mathrm{C}(L)$ is a \color{black} surjetive Banach lattice homomorphism\color{black}. Thus, if $\tilde{J} \subset \mathrm{C}(L)$ is a closed ideal, then $J \defeq P^{-1}(\tilde{J})$ is a closed ideal containing $I$ with $\tilde{J} = P(P^{-1}(\tilde{J})) = J|_L$. It is readily checked that $J$ is the unique closed ideal $H$ containing $I$ with $H|_L = \tilde{J}$\color{black}. Clearly $\tilde{J}$ is $\EuScript{S}_I$-invariant if and only if $J$ is $\EuScript{S}$-invariant.\\
We therefore obtain mutually inverse and inclusion preserving mappings
\begin{align*}
\{J \subset \mathrm{C}(K)\mid J \,\EuScript{S}\textrm{-ideal with } I \subseteq J\} &\leftrightarrow \{\tilde{J} \subset \mathrm{C}(L)\mid \tilde{J}\, \EuScript{S}_I\textrm{-ideal}\},\\
J &\mapsto J|_L\\
P^{-1}(\tilde{J}) &\mapsfrom \tilde{J}.
\end{align*}
\color{black} We now prove that
\begin{align*}
\{p \in \mathrm{Prim}(\EuScript{S})\mid I \subseteq p\} \longrightarrow \mathrm{Prim}(\EuScript{S}_I),\quad 
p \mapsto p|_L
\end{align*}
is a bijective map. Assume \color{black} that $I \subset J = I_\mu$ for some $\mu \in \mathrm{ex}\, \mathrm{P}\color{black}_\EuScript{S}(K)$. Then $\supp \mu \subset L$ and we thus find $\nu\in \mathrm{ex}\,\mathrm{P}_{\EuScript{S}_I}(L)$ with $i(\nu) = \mu$ (see \color{black} Proposition \ref{subsystem}). Moreover we obtain for every $f \in \mathrm{C}(L)$ that
\begin{align}\label{(1)}
\langle |f|,\nu \rangle = \int_L |F| \,\mathrm{d}\mu,
\end{align}
for each extension $F \in \mathrm{C}(K)$ of $f$ to $K$. Thus $f \in I_{\nu}$ if and only if $f \in I_\mu|_L$. If, on the other hand, $\tilde{J} = I_{\nu}$ for some $\nu\in \mathrm{ex}\,\mathrm{P}_{\EuScript{S}_I}(L)$, then Equation \color{black} (\ref{(1)}) \color{black} holds for $\mu = i(\nu)$ and thus $\tilde{J} = I_{\mu}|_L$.\\
Before proceeding with the remaining assertions, we make the following two observations.  
\begin{itemize}
\item For a familiy $(J_\alpha)_{\alpha \in A}$ of $\EuScript{S}$-ideals with $I \subset J_\alpha$ for every $\alpha \in A$
\begin{align*}
\left(\bigcap_{\alpha \in A}J_\alpha\right)|_L = \bigcap_{\alpha \in A}(J_\alpha|_L).
\end{align*}
\item For two $\EuScript{S}$-ideals $J_1, J_2$ with $I \subset J_1,J_2$ the inclusion $J_1|_L \subset J_2|_L$ implies $J_1 \subset J_2$.
\end{itemize}
\color{black} We use these facts to show that
\begin{align*}
\{J \in \mathrm{Rad}(\EuScript{S})\mid I \subseteq J\} \longrightarrow \mathrm{Rad}(\EuScript{S}_I),\quad 
J \mapsto J|_L
\end{align*}
is a well-defined bijection. Take \color{black} an $\EuScript{S}$-ideal $J \subset \mathrm{C}(K)$ with $I \subset J$. Then $J$ is radical if and only if 
\begin{align*}
J = \bigcap_{\substack{p \in \mathrm{Prim}(\EuScript{S})\\ J \subset p}} p 
\end{align*}
which is---\color{black}by the observations above---\color{black}equivalent to 
\begin{align*}
J|_L = \bigcap_{\substack{p \in \mathrm{Prim}(\EuScript{S})\\ J \subset p}} p|_L = \bigcap_{\substack{p \in \mathrm{Prim}(\EuScript{S})\\ J|_L \subset p|_L}} p|_L = \bigcap_{\substack{p \in \mathrm{Prim}(\EuScript{S}_I)\\ J|_L \subset p}} p,
\end{align*}
i.e., $J|_L$ being a radical $\EuScript{S}_I$-ideal.\\
\color{black} Finally, the identity $\mathrm{rad}_{\EuScript{S}_I}(0)= \mathrm{rad}_{\EuScript{S}}(I)|_L$ follows \color{black} from the fact that \color{black} $\mathrm{rad}_{\EuScript{S}_I}(0)$ \color{black} is the smallest radical $\EuScript{S}_I$-ideal and $\mathrm{rad}_{\EuScript{S}}(I)$ is the smallest radical $\EuScript{S}$-ideal containing $I$. 
\end{proof}
\color{black}Our main class of examples for radical ideals are the absolute kernels of (possibly non-ergodic) invariant measures. The following result generalizes Proposition 12 of \cite{Scha1968} using similar arguments.
\begin{proposition}\label{radid1}
The following assertions are valid.
\begin{enumerate}[(i)]
\item For each $\mu \in \mathrm{P}_{\EuScript{S}}(K)$ the $\EuScript{S}$-ideal $I_\mu$ is a radical $\EuScript{S}$-ideal.
\item If $K$ is metrizable, then for each radical $\EuScript{S}$-ideal $I$ there is $\mu \in \mathrm{P}_{\EuScript{S}}(K)$ with $I=I_\mu$.
\end{enumerate}
\end{proposition}
\begin{proof}
For (i) let $\mu \in \mathrm{P}_{\EuScript{S}}(K)$. By Lemma \ref{lemmaquotient} we may assume that $K = \supp \mu$ and it then suffices to show that
$\EuScript{S}$ is radical free. But this directly follows from Remark \ref{capinv}.\smallskip\\
\color{black} We now prove (ii) and assume \color{black} that $K$ is metrizable and $I$ \color{black} is \color{black} a radical $\EuScript{S}$-ideal. We may assume that $I = 0$ (otherwise we pass to $\mathrm{C}(\supp I)$, cf. Lemma \ref{lemmaquotient}). Take a countable base of the topology consisting of nonempty open sets $U_n$, $n \in \N$. Since the supports of ergodic measures are dense in $K$ we find $\mu_n \in \mathrm{ex}\,\mathrm{P}_{\EuScript{S}}(K)$ with $\supp \mu_n \cap U_n \neq \emptyset$ for each $n \in \N$. For
\begin{align*}
\mu \defeq \sum_{n=1}^\infty 2^{-n}\mu_n \in \mathrm{P}_{\EuScript{S}}(K)
\end{align*}
we obtain $\mu(U_n) > 0$ for each $n \in \N$, hence $\mu(U) > 0$ for each non-empty open set $U \subset K$. 
\end{proof}
\begin{remark}\label{nondynamical}
\color{black} Taking $\EuScript{S} = \{\mathrm{Id}\}$ in Proposition \ref{radid1} (ii) yields the probably well-known fact that every compact metric space has a fully supported regular Borel probability measure.\footnotemark[1]\color{black}
\end{remark}
\color{black} The following examples show that part (ii) of Proposition \ref{radid1} is wrong in the non-metric case.\color{black}
\begin{example}\label{nonmetrizable}
\begin{enumerate}[(i)]
\item \color{black}If $K = \Omega \cup \{\infty\}$ is the one-point compactification of an uncountable discrete space $\Omega$, then $\mathrm{C}(K)'$ can be identified with $\ell^1(K)$. Thus there is no fully supported probability measure $\mu \in \mathrm{C}(K)'$.\footnotemark[1]
\footnotetext[1]{\color{black} Remark \ref{nondynamical} and Example \ref{nonmetrizable} (i) were kindly suggested by the referee.\color{black}}
\item \color{black}If $K = \beta \N\setminus \N$ and $\EuScript{S}_\varphi$ is the Markov semigroup induced by the right shift $\varphi$, then 
\begin{align*}
\bigcap_{n \in \N} I_{\mu_n} \not\subset \mathrm{rad}_{\EuScript{S}_\varphi}(0) 
\end{align*}
for every sequence of probability measures $(\mu_n)_{n \in \N} \subset \mathrm{C}(K)'$ (see Corollary 1.10 of \cite{Chou1967}).\color{black}
\end{enumerate}
\end{example}
\section{Centers of Attraction}
\label{sec:4}
Radical $\EuScript{S}$-ideals can also be described via an ergodic stability condition. To formulate our theorem we  write $\overline{\mathrm{co}}\, \EuScript{S}$ for the closed convex hull of $\EuScript{S}$ with respect to the strong operator topology and recall that a net $(T_{\alpha})_{\alpha \in A} \subset \overline{\mathrm{co}}\, \EuScript{S} \subset \mathscr{L}(\mathrm{C}(K))$ \color{black} of operators is \emph{right ergodic} if
\begin{align*} 
\lim_\alpha T_\alpha(\mathrm{Id}-S) = 0
\end{align*}
for each $S \in \EuScript{S}$ with respect to the strong operator topology. We note that there always are right ergodic operator nets for $\EuScript{S}$ (see Corollary 1.5 of \cite{Schr2013a}). \color{black} We give some examples (see Examples 1.2 of \cite{Schr2013a}).
\begin{example}
\begin{enumerate}[(i)]
\item If $\EuScript{S} = \{S^n \mid n \in \N_0\}$ for some Markov operator $S \in \mathscr{L}(\mathrm{C}(K))$, then the \emph{Ces\`{a}ro means}
	\begin{align*}
		C_N \defeq \frac{1}{N}\sum_{n=0}^{N-1} S^n \textrm{ for } N \in \N
	\end{align*}
	define a right ergodic operator sequence $(C_N)_{N \in \N}$ for $\EuScript{S}$. Likewise, the net of \emph{Abel means} $(A_r)_{r \in (0,1)}$ defined by
	\begin{align*}
		A_r \defeq (1-r) \sum_{n=0}^\infty (rS)^n \textrm{ for } r \in (0,1)
	\end{align*}
	is right ergodic for $\EuScript{S}$.
\item If $\EuScript{S} = \{S(t)\mid t \geq 0\}$ is a strongly continuous one-parameter semigroup of Markov operators on $\mathrm{C}(K)$, then the \emph{Ces\`{a}ro means}
	\begin{align*}
		C_Tf \defeq \frac{1}{T}\int_0^T S(t)f\, \mathrm{d}t \textrm{ for } f \in \mathrm{C}(K) \textrm{ and } T > 0
	\end{align*}
	define a right ergodic operator net $(C_T)_{T > 0}$ for $\EuScript{S}$.
\end{enumerate}
\end{example}
\color{black}
The following result generalizes Theorem 4 of \cite{Scha1968}.
\begin{theorem}\label{maintheorem}
For each support $L \subset K$ of an $\EuScript{S}$-ideal
\begin{align*}
\mathrm{rad}_{\EuScript{S}}(I_L) &= \left\{f \in \mathrm{C}(K)\mmid \lim_{\alpha}\int_L T_\alpha |f|\,\mathrm{d}\mu  = 0 \textrm{ for each } \mu \in \mathrm{C}(L)'\right\}\\
&= \left\{f \in \mathrm{C}(K)\mmid \lim_{\alpha} (T_\alpha |f|)|_L = 0 \textrm{ in the norm of } \mathrm{C}(L)\right\}
\end{align*}
where $(T_\alpha)_{\alpha \in A}$ is any right ergodic operator net for $\EuScript{S}$.\\ 
In particular, if $(T_\alpha)_{\alpha \in A}$ is any right ergodic operator net for $\EuScript{S}$, then an $\EuScript{S}$-ideal $I_L$ is a radical $\EuScript{S}$-ideal if and only if every $f \in \mathrm{C}(K)$ satisfying
\begin{align*}
\lim_{\alpha} (T_\alpha |f|)|_L = 0
\end{align*}
vanishes on $L$.
\end{theorem}
\begin{proof}
By Lemma \ref{lemmaquotient} we may assume $L=K$. 
Take $f \in \mathrm{rad}_{\EuScript{S}}(0)$ and any right ergodic operator net $(T_\alpha)_{\alpha \in A}$ for $\EuScript{S}$.\\
Let $\mu \in \mathrm{C}(K)'$ and observe that each subnet of $(T_\alpha'\mu)_{\alpha \in A}$ has a subnet converging to some $\nu \in \mathrm{P}_{\EuScript{S}}(K)$. \color{black} Since $\langle |f|,\nu \rangle = 0$ (see Remark \ref{capinv}), we \color{black} obtain that each subnet of $(\langle T_\alpha |f|,\mu \rangle)_{\alpha \in A}$ has a subnet converging to zero which implies
\begin{align*}
\lim_{\alpha} \langle T_\alpha |f|,\mu \rangle = 0.
\end{align*}
Now let $f \in \mathrm{C}(K)$ with $\lim_{\alpha} T_\alpha |f| = 0$ weakly for some right ergodic operator net $(T_\alpha)_{\alpha \in A}$ for $\EuScript{S}$. Then for $\mu \in \mathrm{P}_{\EuScript{S}}(K)$
\begin{align*}
0 = \lim_{\alpha} \langle T_\alpha |f|,\mu \rangle = \langle |f|,\mu \rangle
\end{align*}
which proves $f \in \mathrm{rad}_{\EuScript{S}}(0)$ and thus the first equation.\\
By Theorem 1.7 of \cite{Schr2013a} the semigroup $\EuScript{S}$ is mean ergodic on $\mathrm{rad}_{\EuScript{S}}(0)$ with mean ergodic projection $P=0$ and therefore
\begin{align*}
\mathrm{rad}_{\EuScript{S}}(0) \subset \left\{f \in \mathrm{C}(K)\mmid \lim_{\alpha} T_\alpha |f| = 0 \textrm{ in the norm of } \mathrm{C}(K)\right\}.
\end{align*}
The converse inclusion is obvious. 
\end{proof}
If $\EuScript{S}$ has a right ergodic operator sequence (for example if it has a F\o lner sequence a\color{black}s defined in Assumption \ref{assumption1} below), then Lebesgue's Theorem yields the following result.
\begin{corollary}\label{cor1}
Suppose that $(T_n)_{n \in \N}$ is a right ergodic operator sequence for $\EuScript{S}$. Then 
\begin{align*}
\mathrm{rad}_{\EuScript{S}}(I) = \left\{f \in \mathrm{C}(K)\mmid \lim_{n \rightarrow \infty} T_n |f|(x)  = 0 \textrm{ for each } x \in \supp I\right\}
\end{align*}
for each $\EuScript{S}$-ideal $I$.
\end{corollary}
 The next corollary shows that if $\EuScript{S}$ is the semigroup generated by a Markov lattice homomorphism $T \in \mathscr{L}(\mathrm{C}(K))$ (i.e., a Koopman operator), the radical $\mathrm{rad}_{\EuScript{S}}(0)$ of the zero ideal coincides with the almost weakly stable part of $\mathrm{C}(K)$ with respect to $T$ as defined in (9.4) on page 176 of \cite{EFHN2015}.\color{black}
\begin{corollary}\label{cor}
Assume that $\varphi\colon K \longrightarrow K$ is a continuous mapping and $\EuScript{S} = \EuScript{S}_\varphi$. Then
\begin{align*}
\mathrm{rad}_{\EuScript{S}}(0)= \left\{f \in \mathrm{C}(K)\mmid \lim_{N \rightarrow \infty}\frac{1}{N}\sum_{n=0}^{N-1} |\langle T_\varphi^nf,\mu \rangle| = 0 \textrm{ for each } \mu \in \mathrm{C}(K)'\right\}.
\end{align*}
\end{corollary}
\begin{proof}
If $f \in \mathrm{rad}_{\EuScript{S}}(0)$ and $\mu \in \mathrm{C}(K)'$, we obtain
\begin{align*}
	\frac{1}{N}\sum_{n=0}^{N-1}|\langle T_\varphi^nf,\mu \rangle| \leq \frac{1}{N}\sum_{n=0}^{N-1}\langle T_\varphi^n|f|,|\mu| \rangle
\end{align*}
for every $N \in \N$ and therefore $\lim_{N \rightarrow \infty} \frac{1}{N}\sum_{n=0}^{N-1}|\langle T_\varphi^nf,\mu \rangle| =0$ by Theorem \ref{maintheorem}. The converse inclusion follows directly from Corollary \ref{cor1}.
\end{proof}
For Koopman semigroups $\EuScript{S}$ \color{black} we also obtain a further \color{black} dynamical characterization of $\mathrm{rad}_{\EuScript{S}}(0)$. \color{black} For the rest of this section we make the following assumption (cf. Examples 1.2 (e) of \cite{Schr2013a}). \color{black}
\begin{assumption}\label{assumption1}
\color{black} Let $\mathcal{S}$ \color{black} be a closed subsemigroup of a locally compact group \color{black} $\mathcal{G}$ \color{black} with left-invariant Haar measure $\lambda$ acting on $K$ such that
\begin{align*}
\color{black} \mathcal{S} \color{black} \times K \longrightarrow K, \quad (s,x) \mapsto sx
\end{align*}
is continuous. Let $\EuScript{S}$ be the associated Koopman semigroup, i.e., \color{black}
\begin{align*}
\EuScript{S} =\{T_s\mid s \in \mathcal{S}\}
\end{align*}
\color{black} with $T_s f(x) \defeq f(sx)$ for $f \in \mathrm{C}(K)$\color{black}, $s \in \mathcal{S}$ \color{black} and $x \in K$, which is strongly continuous by Theorem 4.17 of \cite{EFHN2015}. 
Moreover, we assume that $(F_n)_{n \in \N}$ is a \emph{F\o lner sequence} for \color{black} $\mathcal{S}$\color{black}, i.e., each $F_n$ is a compact subset of \color{black} $\mathcal{S}$ \color{black} with positive measure satisfying
\begin{align*}
\lim_{n \rightarrow \infty} \frac{\lambda(F_n \Delta sF_n)}{\lambda(F_n)} = 0
\end{align*}
for each $s \in \color{black} \mathcal{S} \color{black}$.
\end{assumption}
\color{black}
\begin{example}
\begin{enumerate}[(i)]
\item If $\color{black} \mathcal{S} \color{black}$ is the additive semigroup $\N_0$, then the sequence $(F_n)_{n \in \N}$ defined by $F_n \defeq \{0,...,n-1\}$ for $n \in \N$ is a F\o lner sequence for $\color{black} \mathcal{S} \color{black}$.
\item If $\color{black} \mathcal{S} \color{black}$ is the additive semigroup $\R_{\geq 0}$ and $(t_n)_{n \in \N}$ is any sequence in $(0,\infty)$ with $\lim_{n \rightarrow \infty} t_n = \infty$ then $(F_n)_{n \in \N}$ defined by $F_n \defeq [0,t_n]$ for $n \in \N$ is a F\o lner sequence for $\color{black} \mathcal{S} \color{black}$.
\end{enumerate}
\end{example}
\color{black}

\begin{lemma}
Under Assumption \ref{assumption1} \color{black} $\mathcal{S}$ \color{black} is left amenable and thus $\EuScript{S}$ is right amenable.
Moreover we obtain an ergodic operator sequence $(\EuScript{F}_n)_{n \in \N}$ for $\EuScript{S}$ by setting
\begin{align*}
\EuScript{F}_n f \defeq  \frac{1}{\lambda(F_n)}\int_{F_n}T_{s}f\,\mathrm{d}\lambda(s)
\end{align*}
for $f \in \mathrm{C}(K)$ and $n \in \N$.
\end{lemma}
\begin{proof}
For each $n \in \N$ set 
\begin{align*}
m_n(f) \defeq \frac{1}{\lambda(F_n)}\int_{F_n}f(s)\,\mathrm{d}\lambda(s)
\end{align*}
for $f \in \mathrm{C}_{\mathrm{b}}(\color{black}\mathcal{S}\color{black})$. Then $m_n \in \mathrm{C}_{\mathrm{b}}(\color{black}\mathcal{S}\color{black})'$ with $m_n(\mathbbm{1}) =1$ and $m_n \geq 0$ for each $n \in \N$. Let $m$ be any weak* limit point of $(m_n)_{n \in \N}$. Since
\begin{align*}
\left|\frac{1}{\lambda(F_n)}\int_{F_n}f(s)\,\mathrm{d}\lambda(s) - \frac{1}{\lambda(F_n)}\int_{F_n}f(ts)\,\mathrm{d}\lambda(s)\right| \leq \frac{\lambda(F_n \Delta tF_n)}{\lambda(F_n)} \cdot \|f\|
\end{align*}
for each $f \in \mathrm{C}_{\mathrm{b}}(\color{black}\mathcal{S}\color{black})$, $n \in \N$ and $t \in \color{black}\mathcal{S}\color{black}$, $m$ is an invariant mean. The second assertion is obvious.
\end{proof}
\color{black} We now introduce certain ``attractors'' of the dynamical system $(K;S)$  with respect to the F\o lner sequence $(F_n)_{n \in \N}$.
\begin{definition}
A closed non-empty set $L \subset K$ is a \emph{(global) center of attraction} if for each open set $U \supset L$ we have
\begin{align*}
\lim_{n \rightarrow \infty} \frac{1}{\lambda(F_n)}\lambda(\{s \in F_n\mid sx \in U\}) =1
\end{align*}
for every $x \in K$.
\end{definition}
This type of attraction is quite weak. Loosely speaking, orbits of points may move arbitarily far away from a center of attraction as long as they come back \enquote{often enough} with respect to the F\o lner sequence.\color{black}\\
Global as well as point-dependent centers of attraction for $\mathbb{N}_0$- and $\mathbb{R}_{\geq 0}$-actions have been examined by several authors (see, e.g., \cite{Hilm1936}, \cite{Bern1951}, \cite{JaRo1972}, \cite{Sigm1977}, Exercise I.8.3 in \cite{Mane1987} and \cite{Dai2016}). In a recent paper Z. Chen and X. Dai study the chaotic behavio\color{black}r of minimal centers of attraction with respect to a point for discrete amenable group actions (see \color{black} \cite{ChDa2017}).\\
It is known that in case of $\mathbb{N}_0$-actions on metric compact spaces there always is a unique minimal (global) center of attraction given by the closure of the union of the supports of ergodic measures (see Exercises I.8.3 and II.1.5 in \cite{Mane1987}). The following result shows that this still holds in our more general situation.\color{black}
\begin{theorem}\label{center}
Under Assumption \ref{assumption1} the definition of a center of attraction does not depend on the F\o lner sequence. Moreover, for a closed non-empty set $L \subset K$ the following assertions are equivalent.
\begin{enumerate}[(a)]
\item $L$ is a center of attraction.
\item $I_L\defeq \{f \in \mathrm{C}(K) \mid f|_L = 0\}\subset \mathrm{rad}_{\EuScript{S}}(0)$.
\end{enumerate}
In particular there is a unique minimal center of attraction $\mathrm{M}(\EuScript{S})$ given by the closure of the union of the supports of ergodic measures, i.e.,
\begin{align*}
\mathrm{M}(\EuScript{S}) = \supp \mathrm{rad}_{\EuScript{S}}(0).
\end{align*}
\end{theorem}
\begin{proof}

Take a non-empty and closed set $L \subset K$. The mapping
\begin{align*}
I_L \longrightarrow \mathrm{C}_0(K\setminus L),\quad f \mapsto f|_{K\setminus L}
\end{align*} 
is an isomorphism of Banach lattices. Now $L$ is a center of attraction if and only if
\begin{align*}
\lim_{n \rightarrow \infty} \frac{1}{\lambda(F_n)}\lambda(\{s \in F_n\mid sx \in A\}) =0,
\end{align*}
i.e.,
\begin{align}\label{equ2}
\lim_{n \rightarrow \infty} \frac{1}{\lambda(F_n)}\int_{F_n} \mathbbm{1}_A(sx)\,\mathrm{d}\lambda(s) =0
\end{align}
for each compact set $A \subset K\setminus L$ and each $x \in K$. Since the mapping 
\begin{align*}
\color{black}\mathcal{S}\color{black} \times K \longrightarrow K, \quad (s,x) \mapsto sx
\end{align*}
is continuous, the function $f \colon \color{black}\mathcal{S}\color{black} \times K \longrightarrow \color{black}\C\color{black}, (s,x) \mapsto \mathbbm{1}_A(sx)$ is Borel measurable.
By Lebesgue's and Fubini's theorems Equation \color{black}(\ref{equ2}) \color{black} is thus equivalent to 
\begin{align*}
\lim_{n \rightarrow \infty} \int_K \int_{F_n} \mathbbm{1}_A(sx) \,\mathrm{d}\lambda(s) \, \mathrm{d}\mu(x) = \lim_{n \rightarrow \infty} \frac{1}{\lambda(F_n)}\int_{F_n} T_s'\mu(A)\, \mathrm{d}\lambda(s) =0
\end{align*}
for each $\mu \in \mathrm{C}_0(K\setminus L)'$ and each compact set $A \subset K \setminus L$.
Since the space of compactly supported continuous functions $\mathrm{C}_\mathrm{c}(K\setminus L)$ is dense in $\mathrm{C}_0(K\setminus L)$, this is the case if and only if
\begin{align*}
\lim_{n \rightarrow \infty} \frac{1}{\lambda(F_n)}\int_{F_n} \langle T_s|f|,\mu\rangle \mathrm{d}\lambda(s) =0
\end{align*}
for each $f \in I_L$ and every $\mu \in I_L'$. This means
\begin{align*}
\lim_{n \rightarrow \infty} \EuScript{F}_n|f|=  0
\end{align*}
with respect to the weak topology for each $f \in I_L$, i.e., $I_L \subset \mathrm{rad}_{\EuScript{S}}(0)$.  
\end{proof}

\section{The Primitive Spectrum as a Topological Space}
\label{sec:5}
In this section we return to a general right amenable Markov semigroup $\EuScript{S} \subset \mathscr{L}(\mathrm{C}(K))$ and analyze the topology of $\mathrm{Prim}(\EuScript{S})$. It turns out that it basically has the same properties as the (non-dynamical) primitive spectrum of C*-algebras and the topology of affine schemes of algebraic geometry (cf. Section (2.2) of \cite{GoeWed2010}). We employ methods as in Chapter IV of \cite{Dixm1977} and Section 4.3 of \cite{Pede1979} and first prove two technical lemmas before introducing a topology on $\mathrm{Prim}(\EuScript{S})$. Recall that given $\mu \in \mathrm{P}_\EuScript{S}(K)$ we write $\EuScript{S}_\mu$ for the induced semigroup on $\mathrm{L}^1(K,\mu)$.
\begin{lemma}\label{lemmafixedfunct}
If $\mu \in \mathrm{P}_\EuScript{S}(K)$ and $L \subset K$ is the support of an $\EuScript{S}$-ideal, then $\mathbbm{1}_L \in \fix(\EuScript{S}_{\mu})$.
\end{lemma}
\begin{proof}
We fix $S \in \EuScript{S}$. For each open set $O \supset L$ take a continuous function $f_O$ with $f_O(K) \subset [0,1]$, $f|_L = \mathbbm{1}$ and $f_O|_{(K\setminus O)}=0$. The set $\Lambda$ of open sets containing $L$ is directed  with respect to converse set inclusion \color{black} and thus we obtain a net $(f_O)_{O \in \Lambda}$ with
\begin{align*}
\|\mathbbm{1}_L - f_O\|_{\mathrm{L}^1(K,\mu)} \leq \mu(O\setminus L) \rightarrow 0 
\end{align*}
by regularity of $\mu$. By definition of $\EuScript{S}_\mu$
\begin{align*}
S_{\mu}\mathbbm{1}_L = \lim_{O} Sf_O.
\end{align*}
in $\mathrm{L}^1(K,\mu)$.
Moreover,
\begin{align*}
Sf_O(x) = \langle Sf_O,\delta_x\rangle = \langle f_O,S'\delta_x\rangle = 1
\end{align*}
for each $x \in L$ since $L$ is the support of an $\EuScript{S}$-ideal. This implies
\begin{align*}
0 = \lim_O (\color{black}(\mathbbm{1}-Sf_O)\cdot \mathbbm{1}_L)\color{black} = \mathbbm{1}_L - S_{\mu}\mathbbm{1}_L \cdot \mathbbm{1}_L,
\end{align*}
where the limit is taken in $\mathrm{L}^1(K,\mu)$. Thus $\mathbbm{1}_L = S_{\mu}\mathbbm{1}_L\cdot \mathbbm{1}_L$ which shows $\mathbbm{1}_L \leq S_{\mu}\mathbbm{1}_L$ and consequently $\mathbbm{1}_L \in \fix(\EuScript{S}_\mu)$ by Theorem 13.2 (d) of \cite{EFHN2015}. 
\end{proof}
\begin{lemma}\label{firstlem}
Consider two $\EuScript{S}$-ideals $I_1,I_2$. If $p$ is a primitive $\EuScript{S}$-ideal with $I_1 \cap I_2 \subset p$, then $I_1 \subset p$ or $I_2 \subset p$.
\end{lemma}
\begin{proof}
Let $p = I_\mu$ for some $\mu \in \mathrm{ex}\, \mathrm{P}_{\EuScript{S}}(K)$ and let $L_j \defeq \supp I_j$ for $j=1,2$. By Lemma \ref{lemmafixedfunct}, $\mathbbm{1}_{L_j} \in \fix(\EuScript{S}_\mu)$ for $j=1,2$ and therefore $\mu(L_j) \in \{0,1\}$  since $\mu$ is ergodic. Now $\supp \mu \subset L_1 \cup L_2$ implies \color{black} $\mu(L_1 \cup L_2) = 1$, so there is $j\in \{1,2\}$ with $\mu(L_j) = 1$. But this means $\supp \mu \subset L_j$ and consequently $I_j \subset p$. 
\end{proof}
We are now ready to equip $\mathrm{Prim}(\EuScript{S})$ with a topology \color{black} by defining a Kuratowski closure operator (see page 43 of \cite{Kell1975} for this notion)\color{black}. Recall that the concepts of \emph{hull} and \emph{kernel} have been introduced in Definition \ref{hullker}.\color{black}
\begin{proposition}\label{closureop}
The mapping
\begin{align*}
\overline{\phantom{A}}\colon \mathscr{P}(\mathrm{Prim}(\EuScript{S})) \longrightarrow \mathscr{P}(\mathrm{Prim}(\EuScript{S})), \quad A \mapsto \overline{A}:=\mathrm{hull}(\mathrm{ker}(A))
\end{align*}
defines a Kuratowski closure operator.
\end{proposition}
\begin{proof}
It is readily checked that 
\begin{align*}
\overline{\emptyset} = \emptyset, \, A \subset \overline{A} \textrm{ and } \overline{\overline{A}} = \overline{A}
\end{align*}
for each $A \subset \mathrm{Prim}(\EuScript{S})$. It remains to show that $\overline{A_1 \cup A_2} = \overline{A_1}\cup \overline{A_2}$ for all $A_1,A_2 \subset \mathrm{Prim}(\EuScript{S})$.
Applying Lemma \ref{firstlem} to the ideals $I_j \defeq \mathrm{ker}(A_j)$ for $j=1,2$ yields
\begin{align*}
\overline{A_1 \cup A_2} &= \mathrm{hull}(\mathrm{ker}(A_1 \cup A_2)) = \mathrm{hull}(\mathrm{ker}(A_1)\cap \mathrm{ker}(A_2)) \\
&= \{p \in \mathrm{Prim}(\EuScript{S})\mid \mathrm{ker}(A_1)\cap \mathrm{ker}(A_2) \subset p\}\\
&=  \{p \in \mathrm{Prim}(\EuScript{S})\mid \mathrm{ker}(A_1) \subset p \textrm{ or } \mathrm{ker}(A_2) \subset p\} = \overline{A_1}\cup \overline{A_2}\color{black},
\end{align*}
\color{black} which shows the claim. \color{black}  
\end{proof}
\begin{definition}
The topology on $\mathrm{Prim}(\EuScript{S})$ induced by the closure operator of Proposition \ref{closureop} is called the \emph{hull-kernel-topology}.
\end{definition}
We from now on equip $\mathrm{Prim}(\EuScript{S})$ with the hull-kernel-topology.
\begin{proposition}\label{propertiestop}
The following assertions are valid.
\begin{enumerate}[(i)]
\item The mappings
\begin{align*}
\{\emptyset \neq A \subset \mathrm{Prim}(\EuScript{S})\, \mathrm{closed}\} &\leftrightarrow \mathrm{Rad}(\EuScript{S})\\
A &\mapsto \ker(A)\\
\mathrm{hull}(I) &\mapsfrom I
\end{align*}
are mutually inverse bijections.
\item The sets 
\begin{align*}
U_f\defeq \{p \in \mathrm{Prim}(\EuScript{S})\mid f \notin p\}
\end{align*}
for $f \in \mathrm{C}(K)$ define a base for the hull-kernel-topology of $\mathrm{Prim}(\EuScript{S})$.
\item If $K$ is metrizable, then $\mathrm{Prim}(\EuScript{S})$ has a countable base.
\item The space $\mathrm{Prim}(\EuScript{S})$ is $\mathrm{T}_0$. Given $p \in \mathrm{Prim}(\EuScript{S})$, the set $\{p\}$ is closed if and only if $p$ is a maximal $\EuScript{S}$-ideal.
\item The space $\mathrm{Prim}(\EuScript{S})$ is quasi-compact.
\item The mapping 
\begin{align*}
\pi \colon \mathrm{ex}\,\mathrm{P}_{\EuScript{S}}(K) \longrightarrow \mathrm{Prim}(\EuScript{S}),\quad \mu \mapsto I_{\mu}
\end{align*}
is continuous and surjective. 
\end{enumerate}
\end{proposition}
\begin{proof}
Assertion (i) is obvious. \color{black} For (ii) observe that \color{black} $\mathrm{Prim}(\EuScript{S}) = U_{\mathbbm{1}}$. Now take a closed set $\emptyset \neq A \subset \mathrm{Prim}(\EuScript{S})$. Then $A = \mathrm{ker}(I)$ for some $\EuScript{S}$-ideal $I$ and we obtain
\begin{align*}
\mathrm{Prim}(\EuScript{S})\setminus A = \bigcup_{f \in I} \{p \in \mathrm{Prim}(\EuScript{S})\mid f \notin p\} =  \bigcup_{f \in I} U_f.
\end{align*}
Moreover, each $U_f$ is open since $\mathrm{Prim}(\EuScript{S}) \setminus U_f = \mathrm{hull}(\{f\})$.
This proves (ii)\color{black}.  Assertion (iii) is a direct consequence of (ii).\color{black}\smallskip\\
\color{black} We proceed with part (iv) and prove that $\mathrm{Prim}(\EuScript{S})$ is a $\mathrm{T}_0$-space. \color{black} If $p_1, p_2 \in \mathrm{Prim}(\EuScript{S})$ with $p_1\neq  p_2$, then  $M_1 \neq M_2$ for the supports $M_i \defeq \supp p_i$, $i=1,2$. We may assume that there is $x \in M_2 \setminus M_1$ and find $f \in \mathrm{C}(K)$ with $f|_{M_1} = 0$ and $f(x) =1$. Then $p_1 \notin U_{f}$ and $p_2 \in U_{f}$.\\
\color{black} For the second part of (iv) take \color{black} a maximal $\EuScript{S}$-ideal and assume that $p \in \overline{\{m\}}$. Then $m \subset p$ and thus $m=p$ by maximality of $m$.\\
Conversely, suppose that $\{m\}$ is closed and take a maximal $\EuScript{S}$-ideal $p$ with $m \subset p$. Then $\mathrm{ker}(\{m\}) = m \subset p$ and thus
\begin{align*}
p \in \mathrm{hull}(\mathrm{ker}(\{m\})) = \overline{\{m\}},
\end{align*}
i.e., $p = m$.\smallskip\\
\color{black} For the proof of (v) take \color{black} closed subsets $A_j \subset \mathrm{Prim}(\EuScript{S})$ for $j \in J$ with
\begin{align*}
\bigcap_{j \in J} A_j = \emptyset
\end{align*}
and let $I_j\defeq \mathrm{ker}(A_j)$ be the corresponding radical ideals for $j\in J$. We show that
\begin{align*}
\sum_{j \in J} I_j = \mathrm{C}(K).
\end{align*}
Denote the ideal on the left side by $I$ and assume that it is a proper invariant ideal. Since there are no dense ideals in $\mathrm{C}(K)$, the closure $\overline{I}$ is contained in a maximal $\EuScript{S}$-ideal $p$. But then $p \in A_j$ for each $j \in J$ since the sets $A_j$ are closed, a contradiction.\\
Take $j_1,...,j_k$ with $1 \in I_{j_1} + ... + I_{j_k}$ for some $k \in \N$. Then
\begin{align*}
\sum_{m =1}^k I_{j_m} = \mathrm{C}(K)
\end{align*}
and consequently 
\begin{align*}
\bigcap_{m=1}^k A_{j_m} = \emptyset.
\end{align*}
\smallskip\\
\color{black} Finally, assertion (vi) follows from the fact that the set
\begin{align*}
\pi^{-1}(U_f) = \{\mu \in \mathrm{ex}\,\mathrm{P}_{\EuScript{S}}(K)\mid \langle |f|,\mu\rangle \neq 0\}
\end{align*} 
is open in $\mathrm{ex}\,\mathrm{P}_{\EuScript{S}}(K)$ for each $f \in \mathrm{C}(K)$.\color{black} 
\end{proof}
\color{black}
\begin{corollary}\label{easyconv}
A net $(p_\alpha)_{\alpha \in A}$ in $\mathrm{Prim}(\EuScript{S})$ converges to $p \in \mathrm{Prim}(\EuScript{S})$ if and only if for each open set $U \subset K$ with $\supp p \cap U \neq \emptyset$ there is $\alpha_0 \in A$ with $\supp p_\alpha \cap U \neq \emptyset$ for every $\alpha \geq \alpha_0$.
\end{corollary}
\begin{proof}
Consider the sets $V_f \defeq \{x \in K \mid f(x) \neq 0\}$ for $f \in \mathrm{C}(K)$. Proposition \ref{propertiestop} (ii) shows that a net $(p_\alpha)_{\alpha \in A}$ in $\mathrm{Prim}(\EuScript{S})$ converges to $p \in \mathrm{Prim}(\EuScript{S})$ if and only if for each $f \in \mathrm{C}(K)$ with $f|_{\supp p} \neq 0$ there is $\alpha_0 \in A$ with $f|_{\supp(p_\alpha)} \neq 0$ for every $\alpha \geq \alpha_0$, i.e., for each $f \in \mathrm{C}(K)$ with $V_f \cap \supp p \neq \emptyset$ there is $\alpha_0$ with $V_f \cap \supp p_\alpha \neq \emptyset$ for every $\alpha \geq \alpha_0$. Since the sets $V_f$ are a base of the topology of $K$, this shows the claim.
\end{proof}
\color{black}
\begin{example}\label{PrimExamples}
\begin{enumerate}[(i)]
\item For the trivial semigroup $\EuScript{S} = \{\mathrm{Id}\}$ every ideal is invariant and $\mathrm{Prim}(\EuScript{S})$ coincides with the maximal ideal space of the commutative $\mathrm{C}^*$-algebra $\mathrm{C}(K)$, i.e., it is homeomorphic to $K$.
\item Consider the torus $K = \T =\color{black} \{z \in \C\mid |z| = 1\}$ and the rotation $\varphi_a(z):= az$ for $z \in \T$ and some fixed $a \in \T$ with $a^k = 1$  for some $k \in \N$. \color{black} Denote the group of $k$th roots of unity by $G_k$. \color{black} The ergodic measures are then precisely the measures $\mu_b \in \mathrm{P}_{\EuScript{S}_\varphi}(K)$ with $\mu_b \defeq \frac{1}{k}\sum_{j=0}^{k-1}\delta_{a^jb}$ for $b \in \T$. Their supports are clearly the sets
\begin{align*}
M_b\defeq bG_k=\{bz  \mid z \in \T, \, z^k = 1\} 
\end{align*}
for $b \in \T$. Using Corollary \ref{easyconv} a moment's thought reveals that 
\begin{align*}
\T/G_k \longrightarrow \mathrm{Prim}(\EuScript{S}_\varphi), \quad bG_k \mapsto I_{M_b}
\end{align*}
is a homeomorphism \color{black} if we endow the factor group $\T/G_k$ with the quotient topology.
\item Consider the space $L\defeq \{0,1\}^\N$ and the shift $\varphi\colon L \longrightarrow L$ given by $\varphi((x_n)_{n \in \N})\defeq (x_{n+1})_{n \in \N}$ for each $(x_n)_{n \in \N} \in L$.
For each $k \in \N$ consider the minimal non-empty closed invariant \color{black} set
\begin{align*}
M_k \defeq \{\varphi^n(x^k)\mid n \in \{0,...,2k-1\}\}
\end{align*}
with $x^k=(x_m^k)_{m \in \N}$ defined by
\begin{align*}
x_m^k \defeq \begin{cases} 0 & \textrm{if } m \in \{1,...,k\}+ 2k \N_0,\\
1 & \textrm{else}.
\end{cases}
\end{align*}
Now  if
\begin{align*}
K \defeq \overline{\bigcup_{k \in \N}M_k},
\end{align*}
then it is readily seen that $K$ is the invariant set
\begin{align*}
\bigcup_{k \in \N}M_k \cup \{(x_m)_{m \in \N} \in L\mid (x_m)_{m \in \N} \textrm{ increasing or decreasing}\}.
\end{align*}
We restrict $\varphi$ to $K$ and \color{black} claim that $\mathrm{Prim}(\EuScript{S}_\varphi\color{black})$ is not Hausdorff. It suffices to show that the sequence $(m_n)_{n \in \N}$ in $\mathrm{Prim}(\EuScript{S}_\varphi\color{black})$ with $m_n \defeq I_{M_n}$ converges to two different points.\\
To this end, consider $k \in \N$ and the open subset
\begin{align*}
U \defeq \left(\prod_{i=1}^k \{1\} \times \prod_{i=k+1}^\infty \{0,1\}\right) \cap K.
\end{align*}
\color{black} Then $M_l \cap U \neq \emptyset$ for each $l \geq k$. \color{black} By Remark \ref{easyconv} \color{black} this implies $m_l \rightarrow I_{\{(1)_{n \in \N}\}}\color{black}$ and a similar argument shows $m_l \rightarrow I_{\{(0)_{n \in \N}\}}\color{black}$.
\end{enumerate}
\end{example}

\section{Continuous Functions on the Primitive Spectrum}
\label{sec:6}
It is our goal to describe the continuous functions on $\mathrm{Prim}(\EuScript{S})$. As above we write $\mathrm{M}(\EuScript{S})$ for the support of $\mathrm{rad}_{\EuScript{S}}(0)$, i.e., the closure of the union of all supports of invariant ergodic measures, and recall that the semigroup on $\mathrm{C}(\mathrm{M}(\EuScript{S}))$ induced by $\EuScript{S}$ is denoted by $\EuScript{S}_{\mathrm{rad}_{\EuScript{S}}(0)}$. Now consider the following functions.
\begin{definition}
For a function $f \in \fix(\EuScript{S}_{\mathrm{rad}_{\EuScript{S}}(0)})$ we define 
\begin{align*}
\hat{f}\colon \mathrm{Prim}(\EuScript{S}) \longrightarrow \C, \quad I_\mu \mapsto \int_{\mathrm{M}(\EuScript{S})} f\, \mathrm{d}\mu.
\end{align*}
\end{definition}

Note that each $f \in \fix(\EuScript{S}_{\mathrm{rad}_{\EuScript{S}}(0)})$ is constant on supports of ergodic measures and therefore $\int_{\mathrm{M}(\EuScript{S})} f\, \mathrm{d}\mu$ only depends on $I_\mu$ and \color{black} not on $\mu$ itself\color{black}. Thus $\hat{f}$ is in fact well-defined for each $f \in \fix(\EuScript{S}_{\mathrm{rad}_{\EuScript{S}}(0)})$ and the next lemma shows that it is even continuous.\color{black}
\begin{lemma}
If $f \in \fix(\EuScript{S}_{\mathrm{rad}_{\EuScript{S}}(0)})$, \color{black} then $\hat{f} \in \mathrm{C}(\mathrm{Prim}(\EuScript{S}))$.
\end{lemma}
\begin{proof}
Let $p = I_{\mu} \in \mathrm{Prim}(\EuScript{S})$ and $\varepsilon > 0$. We set
\begin{align*}
f_{\varepsilon}\defeq \sup\left(\varepsilon \cdot \mathbbm{1} - \left|f-\int_{\mathrm{M}(\EuScript{S})} f\, \mathrm{d}\mu \cdot \mathbbm{1}\right|,0\right)\Bigr|_{\mathrm{M}(\EuScript{S})} \in \mathrm{C}(\mathrm{M}(\EuScript{S})).
\end{align*}
Then $U\defeq U_{f_{\varepsilon}}$ is an open neighborhood of $p$. Moreover, for each $q = I_{\nu} \in U$
\begin{align*}
\varepsilon - \left|\int_{\mathrm{M}(\EuScript{S})} f\, \mathrm{d}\nu - \int_{\mathrm{M}(\EuScript{S})} f\, \mathrm{d}\mu\right| > 0
\end{align*}
which means $|\hat{f}(p) - \hat{f}(q)| < \varepsilon$. 
\end{proof}
In the radical free case (i.e., $\mathrm{M}(\EuScript{S}) = K$) we now obtain a linear mapping from the fixed space $\fix(\EuScript{S})$ to $\mathrm{C}(\mathrm{Prim}(\EuScript{S}))$. It turns out that this is actually an isomorphism.
\begin{theorem}\label{contfunct1}
If $\EuScript{S}$ is radical free, then the fixed space $\fix(\EuScript{S})$ is a Banach sublattice of $\mathrm{C}(K)$ and the mapping
\begin{align*}
\hat{\,}\colon\fix(\EuScript{S}) \longrightarrow \mathrm{C}(\mathrm{Prim}(\EuScript{S})), \quad f \mapsto \hat{f}
\end{align*}
is an isometric  Markov \color{black} lattice isomorphism. 
\end{theorem}
\begin{proof}
We first show that $\fix(\EuScript{S})$ is a sublattice of $\mathrm{C}(K)$. Take $f \in \fix(\EuScript{S})$ and $S \in \EuScript{S}$. Since $\mathrm{M}(\EuScript{S})= K$, it suffices to prove that $S_p|f||_{\supp p} = |f||_{\supp p}$ for each $p \in \mathrm{Prim}(\EuScript{S})$. However, this is true since the fixed space $\fix(\EuScript{S}_p)$ consists only of constant functions for every $p \in \mathrm{Prim}(\EuScript{S})$.\medskip\\
The mapping $\,\hat{\,}\,$ is clearly linear  and $\hat{\mathbbm{1}} = \mathbbm{1}$\color{black}. \color{black} Next we show that $\,\hat{\,}\,$ is an isometric Markov lattice homomorphism. \color{black} For $f \in \fix(\EuScript{S})$
\begin{align*}
\|\hat{f}\|= \sup_{p \in \mathrm{Prim}(\EuScript{S})} |\hat{f}(p)| = \sup_{\mu \in \mathrm{ex}\,\mathrm{P}_{\EuScript{S}}(K)} |\langle f, \mu \rangle| \leq \|f\|.
\end{align*}
The set $\{x \in K\mid |f(x)| = \|f\|\}$ is the support of an $\EuScript{S}$-ideal \color{black}(\color{black}see \color{black} Theorem 1.2 in \cite{Sine1968}) and thus contains a minimal support of an $\EuScript{S}$-ideal $M$ which in turn supports an ergodic measure $\mu$. This implies
\begin{align*}
|\hat{f}(I_M)| = |\langle f,\mu \rangle| = |f(x)| = \|f\|
\end{align*}
for each $x \in M$ and consequently $\|\hat{f}\| = \|f\|$.\\
Now take $f \in \fix(\EuScript{S})$ and $p = I_\mu \in \mathrm{Prim}(\EuScript{S})$. For each $x \in \supp(\mu)$
\begin{align*}
|\hat{f}(p)| = |f(x)| = |f|(x) = \widehat{|f|}(p).\color{black}
\end{align*}
It remains to show that  $\,\hat{\,}\,$ is surjective. Take $f \in \mathrm{C}(\mathrm{Prim}(\EuScript{S}))$ with $0 \leq f \leq \mathbbm{1}$. We fix $n \in \N$ and consider the open sets
\begin{align*}
U_{k,n} \defeq \left\{p \in \mathrm{Prim}(\EuScript{S})\mmid \frac{k-1}{n} < f(p) < \frac{k+1}{n}\right\}
\end{align*}
for $k \in \{0,...,n\}$. Then $U_{k,n}^c = \mathrm{hull}(I_{k,n}\color{black})$ for invariant ideals $I_{k,n} \subset \mathrm{C}(K)$ and $k \in \{0,...,n\}$. Assume
\begin{align*}
I\defeq \sum_{k=0}^n I_{k,n}\color{black} \neq \mathrm{C}(K).
\end{align*}
Then $I$ is contained in a maximal $\EuScript{S}$-ideal $p$. Since $I_{k,n}\color{black} \subset p$ for all $k \in \{0,...,n\}$,
\begin{align*}
p \in \bigcap_{k=0}^n \mathrm{hull}(I_{k,n}\color{black}) = \left(\bigcup_{k=0}^n U_{k,n}\right)^c = \emptyset,
\end{align*}
a contradiction.\\
We thus find $0 \leq f_{k,n}\color{black} \in I_{n,k}$ for $k\in \{0,...,n\}$ with $\mathbbm{1} = \sum_{k=0}^n f_{k,n}\color{black}$ ( see \color{black} II.5.1.4 in \cite{Blac2006}). Now set $g_n \defeq \sum_{k=1}^n\frac{k}{n}f_{k,n}\color{black}$.\\
Take $p \in \mathrm{Prim}(\EuScript{S})$. If $k \in \{0,...,n\}$ with $p \notin U_{k,n}$,  then $I_{k,n}\color{black} \subset p$ and therefore $f_{k,n}\color{black} \in p$, i.e., $f_{k,n}\color{black}|_{\supp p} = 0$. This implies 
\begin{align*}
|g_n(x) - f(p)| &= \left| \sum_{k\colon p \in U_{k,n}} \frac{k}{n}f_{k,n}\color{black}(x) - \sum_{k\colon p \in U_{k,n}} f_{k,n}\color{black}(x)f(p)\right|\\
&\leq \sum_{k\colon p \in U_{k,n}}\left|\frac{k}{n} - f(p)\right||f_{k,n}\color{black}(x)| \leq \frac{1}{n}
\end{align*}
for $x \in \mathrm{supp}\, p$. In particular we obtain
\begin{align*}
|g_n(x) - g_m(x)| \leq \frac{1}{n}+\frac{1}{m}
\end{align*}
for all $x \in M\defeq \bigcup_{p \in \mathrm{Prim}(\EuScript{S})}\supp\,p$ and all $n,m \in \N$. Since $\EuScript{S}$ is radical free, $M$ is dense whence $(g_n)_{n \in \N}$ is a Cauchy sequence in $\mathrm{C}(K)$. Denoting its limit by $g$ we obtain $g(x) = f(p)$ for $x \in \supp \, p$, $p \in \mathrm{Prim}(\EuScript{S})$ and, since $M$ is dense, $g \in \fix(\EuScript{S})$. Moreover, we clearly have $\hat{g} = f$. 
\end{proof}
The first part of the proof of Theorem \ref{contfunct1} is based on the proof of Theorem 5 in \cite{Scha1968}, while the second part uses arguments from the Dauns-Hofmann Theorem II.6.5.10 in \cite{Blac2006}.\\
We now focus on the general case, i.e., $\EuScript{S}$ not being radical free.
\begin{lemma}\label{lemmahomeo}
The mapping
\begin{align*}
\theta \colon \mathrm{Prim}(\EuScript{S}) \longrightarrow \mathrm{Prim}(\EuScript{S}_{\mathrm{rad}_{\EuScript{S}}(0)}),\quad p \mapsto p|_{\mathrm{M}(\EuScript{S})}
\end{align*}
is a homeomorphism.
\end{lemma}
\begin{proof}
Note first that $\theta$ is well-defined and bijective by Lemma \ref{lemmaquotient} since $\mathrm{rad}_{\EuScript{S}}(0) \subset p$ for each $p \in \mathrm{Prim}(\EuScript{S})$. Lemma \ref{lemmaquotient} also implies that 
\begin{align*}
\mathrm{Rad}(\EuScript{S}) \longrightarrow \mathrm{Rad}(\EuScript{S}_{\mathrm{rad}_{\EuScript{S}}(0)}),\quad I \mapsto I|_{\mathrm{M}(\EuScript{S})}
\end{align*}
is bijective. By Proposition \ref{propertiestop}  (i) \color{black} we obtain that $A \subset \mathrm{Prim}(\EuScript{S})$ is closed if and only if $A= \mathrm{hull}(I)$ for some $I \in \mathrm{Rad}(\EuScript{S})$. If $A$ is closed, we therefore obtain \color{black}
\begin{align*}
\theta(A) &= \theta(\{p \in \mathrm{Prim}(\EuScript{S})\mid I \subset p\})\\
&= \{q \in \mathrm{Prim}(\EuScript{S}_{\mathrm{rad}_{\EuScript{S}}(0)})\mid I \subset \theta^{-1}(q)\}\\
&= \{q \in \mathrm{Prim}(\EuScript{S}_{\mathrm{rad}_{\EuScript{S}}(0)})\mid I|_{\mathrm{M}(\EuScript{S})} \subset q\}\\
&= \mathrm{hull}(I|_{\mathrm{M}(\EuScript{S})}),
\end{align*}
 and therefore $\theta(A)$ is closed. Conversely, if $\theta(A)$ is closed, then by Lemma \ref{lemmaquotient} and Proposition \ref{propertiestop} (i) there is $I \in \mathrm{Rad}(\EuScript{S})$ with $\theta(A) = \mathrm{hull}(I|_{\mathrm{M}(\EuScript{S})})$ and, by the above, we obtain $A= \mathrm{hull}(I)$. 
\end{proof}
By combining Lemma \ref{lemmahomeo} with Theorem \ref{contfunct1} we obtain  the main result of this section.\color{black}
\begin{theorem}\label{contfunct}
The fixed space $\fix(\EuScript{S}_{\mathrm{rad}_{\EuScript{S}}(0)})$ is a Banach sublattice of $\mathrm{C}(\mathrm{M}(\EuScript{S}))$ and the mapping
\begin{align*}
\hat{\,}\colon\fix(\EuScript{S}_{\mathrm{rad}_{\EuScript{S}}(0)}) \longrightarrow \mathrm{C}(\mathrm{Prim}(\EuScript{S})), \quad f \mapsto \hat{f}
\end{align*}
is an isometric Markov \color{black} lattice isomorphism. 
\end{theorem}
\section{Mean Ergodic Semigroups of Markov Operators}
\label{sec:7}
Using \color{black} the space $\mathrm{C}(\mathrm{Prim}(\EuScript{S}))$ we can now analyze mean ergodicity of Markov semigroups \color{black} and extend \color{black} Theorem 2 of \cite{Scha1967}. \color{black} Recall that $\overline{\mathrm{co}}\,\EuScript{S}$ denotes the closed convex hull of $\EuScript{S}$ with respect to the strong operator topology. The \color{black} right amenable semigroup $\EuScript{S}$ is \emph{mean ergodic} if there is $P \in \overline{\mathrm{co}}\,\EuScript{S}$ with $PS=SP=P$ for each $S \in \EuScript{S}$ (see \cite{Nage1973} or \cite{Schr2013a} for this concept). In this case $P$ is unique with these properties and a projection onto the fixed space $\fix(\EuScript{S})$ of $\EuScript{S}$, called the \emph{mean ergodic projection}.\color{black}
\begin{theorem}\label{meanergodiccor2}
The following assertions are equivalent.
\begin{enumerate}[(a)]
\item $\EuScript{S}$ is mean ergodic.
\item The following two conditions are satisfied.
\begin{enumerate}[(i)]
\item The mapping
\begin{align*}
\mathrm{ex}\, \mathrm{P}_{\EuScript{S}}(K) \longrightarrow \mathrm{Prim}(\EuScript{S}), \quad \mu \mapsto I_\mu
\end{align*}
is a homeomorphism.
\item For each $f \in \fix(\EuScript{S}_{\mathrm{rad}_{\EuScript{S}}(0)})$ there is $F \in \fix(\EuScript{S})$ with $f = F|_{\mathrm{M}(\EuScript{S})}$.
\end{enumerate}
\item The following three conditions are satisfied.
\begin{enumerate}[(i)]
\item The primitive spectrum $\mathrm{Prim}(\EuScript{S})$ is a Hausdorff space.
\item For each $\mu \in \mathrm{ex}\,\mathrm{P}_{\EuScript{S}}(K)$ \color{black} the support $\supp \mu$ is uniquely ergodic, i.e., $\mu$ is the only invariant measure having its support in $\supp \mu$.
\item For each $f \in \fix(\EuScript{S}_{\mathrm{rad}_{\EuScript{S}}(0)})$ there is  $F \in \fix(\EuScript{S})$ with $f = F|_{\mathrm{M}(\EuScript{S})}$.
\end{enumerate}
\end{enumerate}
\begin{remark}\label{meanprim2}
 Note that assertion (c) (i) of \cref{meanergodiccor2} implies that each primitive ideal is maximal (see Proposition \ref{propertiestop} (iv)). This shows that the concept of maximal $\EuScript{S}$-ideals is sufficient for mean ergodic semigroups. In particular, combining \cref{meanergodiccor2} with Proposition \ref{radid1} (i) implies Theorem 2 of \cite{Scha1967}.\color{black}
\end{remark}
\end{theorem}
\begin{proof}[ of \cref{meanergodiccor2}]
\enquote{(a) $\Rightarrow$ (c)}: \color{black}Assume that $\EuScript{S}$ is mean ergodic with mean ergodic projection $P \in \mathscr{L}(\mathrm{C}(K))$. We first show that $\mathrm{Prim}(\EuScript{S})$ is Hausdorff.\\
Consider $I_{\mu_1}$, $I_{\mu_2} \in \mathrm{Prim}(\EuScript{S})$ with $\mu_1 \neq \mu_2$. Since $\EuScript{S}$ is mean ergodic, we find $f \in \fix(\EuScript{S})$ with
\begin{align*}
c_1 \coloneqq \langle f,\mu_1\rangle < \langle f, \mu_2\rangle \eqqcolon c_2
\end{align*}
 by Theorem 1.7 of \cite{Schr2013a}. \color{black}
Choose $c \in (c_1,c_2)$ and set $U_1 \defeq f^{-1}((-\infty,c))$ and $U_2 \defeq f^{-1}((c,\infty))$. The sets
\begin{align*}
V_i \defeq \{p \in \mathrm{Prim}(\EuScript{S})\mid \supp p \cap U_i \neq \emptyset\} \subset \mathrm{Prim}(\EuScript{S})
\end{align*}
are open by Proposition \ref{propertiestop} (ii) and $I_{\mu_i} \in V_i$ for $i=1,2$.\\
Assume there is $p \in V_1\cap V_2$. Then there are $x_i \in U_i \cap \supp p$ for $i=1,2$ and thus $f(x_1)< c < f(x_2)$. Since $f$ is constant on supports of ergodic measures, this is a contradiction. \\
Given $\mu \in \mathrm{ex} \, \mathrm{P}_{\EuScript{S}}(K)$ we know that $\EuScript{S}_{I_\mu} \subset \mathscr{L}(\mathrm{C}(\supp \mu))$ is also mean ergodic. Since $\fix(\EuScript{S}_{I_\mu})$ is one-dimensional, we obtain that $\fix(\EuScript{S}_{I_\mu}')$ is one dimensional, too. Thus, the supports of ergodic measures are uniquely ergodic.\\
Next, take $f \in \fix(\EuScript{S}_{\mathrm{rad}_{\EuScript{S}}(0)})$ and let $G$ be any continuous extension of $f$ to $K$. Then $F \defeq PG \in \fix(\EuScript{S})$ with $F|_{\mathrm{M}(\EuScript{S})} = PG|_{\mathrm{M}(\EuScript{S})} = f$. Thus (a) implies (c).\smallskip\\
\enquote{(c) $\Rightarrow$ (b)}: \color{black} Now suppose that (i), (ii) and (iii) of (c) are valid. We first show that  $\mathrm{ex}\,\mathrm{P}_{\EuScript{S}}(K)$ is compact. \color{black} Take a net $(\mu_{\alpha})_{\alpha \in A}$ in $\mathrm{ex}\,\mathrm{P}_{\EuScript{S}}(K)$ with $\lim_{\alpha} \mu_{\alpha} = \mu \in \mathrm{P}_{\EuScript{S}}(K)$. Since $I_\mu$ is a radical ideal  by Proposition \ref{radid1} (i) \color{black} we obtain
\begin{align*}
I_\mu = \bigcap_{\substack{p \in \mathrm{Prim}(\EuScript{S})\\I_\mu \subset p}} p.
\end{align*}
Now take $p \in \mathrm{Prim}(\EuScript{S})$ with $p \supset I_\mu$ and any $f \in \mathrm{C}(K)$ with $p \in U_f$. Then $f \notin p$ and consequently $f \notin I_\mu$. This implies $\langle |f|,\mu\rangle \neq 0$ and thus there is $\alpha_0 \in A$ with $\langle |f|,\mu_{\alpha}\rangle \neq 0$ for all $\alpha \geq \alpha_0$.  But then $I_{\mu_\alpha} \rightarrow p$ and, since $\mathrm{Prim}(\EuScript{S})$ is Hausdorff, this implies that there is only one such $p$, hence  $I_\mu$ is primitive. \color{black} Applying (ii)  shows \color{black} that $\mu$ is ergodic.\\
 We now obtain that 
\begin{align*}
\pi \colon \mathrm{ex}\, \mathrm{P}_{\EuScript{S}}(K) \longrightarrow \mathrm{Prim}(\EuScript{S}), \quad \mu \mapsto I_\mu
\end{align*}
is a homeomorphism since the mapping is injective by (ii), $\mathrm{ex}\, \mathrm{P}_{\EuScript{S}}(K)$ is compact and $\mathrm{Prim}(\EuScript{S})$ is Hausdorff by (i).\color{black}
\smallskip\\
\enquote{(b) $\Rightarrow$ (a)}: \color{black} Finally assume that (b) is valid. The mapping 
\begin{align*}
\Phi_1 \colon \mathrm{C}(\mathrm{Prim}(\EuScript{S})) \longrightarrow \mathrm{C}(\mathrm{ex}\,\mathrm{P}_{\EuScript{S}}(K))\color{black}, \quad f \mapsto f\circ \pi
\end{align*}
is then an isometric  Markov lattice isomorphism \color{black} and by Theorem \ref{contfunct} the mapping 
\begin{align*}
\hat{ } \, \colon \fix(\EuScript{S}_{\mathrm{rad}_{\EuScript{S}}(0)}) \longrightarrow \mathrm{C}(\mathrm{Prim}(\EuScript{S}))
\end{align*}
is \color{black} so, too\color{black}. Now consider the map
\begin{align*}
\Phi_2 \colon \fix(\EuScript{S}) \longrightarrow \fix(\EuScript{S}_{\mathrm{rad}_{\EuScript{S}}(0)}), \quad f \mapsto f|_{\mathrm{M}(\EuScript{S})}.
\end{align*}
This is an isometric Banach space embedding (isometry follows with the same arguments as in the proof of Theorem \ref{contfunct}) and by (ii) it is surjective. Thus
\begin{align*}
\Phi \defeq \Phi_1 \circ \, \hat{ }\, \circ \Phi_2 \colon \fix(\EuScript{S}) \longrightarrow \mathrm{C}(\mathrm{ex}\, \mathrm{P}_{\EuScript{S}}(K)\color{black}),\quad f \mapsto \langle f,\,\cdot\,\rangle
\end{align*}
is an isometric isomorphism of Banach spaces.\\
Now take $\mu_1,\mu_2 \in \mathrm{P}_{\EuScript{S}}(K)$ with $\mu_1\neq \mu_2$. The space $\mathrm{ex}\, \mathrm{P}_{\EuScript{S}}(K)$ \color{black} is compact since it is homeomorphic to $\mathrm{Prim}(\EuScript{S})$. By Choquet theory (see \color{black} Proposition 1.2 in \cite{Phel2001}) we thus find measures $\tilde{\mu}_1,\tilde{\mu}_2 \in \mathrm{C}(\mathrm{ex}\, \mathrm{P}_{\EuScript{S}}(K)\color{black})'$ with $\tilde{\mu}_1 \neq \tilde{\mu}_2$ such that
\begin{align*}
\langle f,\mu_i \rangle = \int_{\mathrm{ex}\,\mathrm{P}_{\EuScript{S}}(K)} \langle f,\nu \rangle \, \mathrm{d}\tilde{\mu}_i(\nu)
\end{align*}
for each $f \in \mathrm{C}(K)$ and $i=1,2$. 
We then obtain
\begin{align*}
\langle f,\mu_i \rangle = \int_{\mathrm{ex}\,\mathrm{P}_{\EuScript{S}}(K)}\Phi(f)(\nu) \, \mathrm{d}\tilde{\mu}_i(\nu) = \langle \Phi(f),\tilde{\mu}_i\rangle
\end{align*}
for each $f \in \fix(\EuScript{S})$ and $i=1,2$. Since $\mathrm{C}(\mathrm{ex}\, \mathrm{P}_{\EuScript{S}}(K))$ separates $\mathrm{C}(\mathrm{ex}\, \mathrm{P}_{\EuScript{S}}(K))'$, this proves that $\fix(\EuScript{S})$ separates $\mathrm{P}_{\EuScript{S}}(K)$. Now $\EuScript{S}$ consists of Markov operators and therefore $\fix(\EuScript{S})$ separates $\fix(\EuScript{S}')$. Thus $\EuScript{S}$ is mean ergodic by Theorem 1.7 of \cite{Schr2013a}.\color{black}
 
\end{proof}
\begin{corollary}\label{meanergodiccor}
If $\EuScript{S}$ is radical free, then $\EuScript{S}$ is mean ergodic if and only if
\begin{align*}
\mathrm{ex}\, \mathrm{P}_{\EuScript{S}}(K) \longrightarrow \mathrm{Prim}(\EuScript{S}), \quad \mu \mapsto I_\mu
\end{align*}
is a homeomorphism. 
\end{corollary}
The next corollary follows from Proposition \ref{subsystem} and Lemma \ref{lemmahomeo}.
\begin{corollary}\label{cor2}
The semigroup $\EuScript{S}$ is mean ergodic if and only if $\EuScript{S}_{\mathrm{rad}_{\EuScript{S}}(0)}$ is mean ergodic and for each $f \in \fix(\EuScript{S}_{\mathrm{rad}_{\EuScript{S}}(0)})$ there is $F \in \fix(\EuScript{S})$ with $f = F|_{\mathrm{M}(\EuScript{S})}$.
\end{corollary}
Finally we discuss some examples showing that the conditions of Theorem \ref{meanergodiccor2} (c) are independent of each other.
\begin{example}\label{finalex}
Consider the following continuous mappings $\varphi \colon K \longrightarrow K$ and the induced semigroups $\EuScript{S} = \EuScript{S}_\varphi  \subset \mathscr{L}(\mathrm{C}(K))$.\color{black}
\begin{enumerate}[(i)]
\item If $K = [0,1]$ and $\varphi(x) = x^2$ for $x \in K$, then $\mathrm{M}(\EuScript{S}) = \{0,1\}$ and the primitive spectrum is the two point discrete space. Clearly, both fixed points define uniquely ergodic sets, so conditions (i) and (ii) of Theorem \ref{meanergodiccor2} (c) are fulfilled. However, $\EuScript{S}$ is not mean ergodic since the function $f\colon \{0,1\}\longrightarrow \C$ defined by $f(0) \defeq 0$ and $f(1) \defeq 1$ has no invariant continuous extension to $K$.
\item If $K = \T$ then there is a homeomorphism $\varphi \colon \T\longrightarrow \T$ such that $\EuScript{S}$ is not uniquely ergodic, but minimal (see Theorem 5.8 of \cite{Parr1981}). \color{black} In particular, the support of every ergodic measure is $K$. \color{black} Thus the primitive spectrum is trivial and $\EuScript{S}$ is radical free whence Theorem \ref{meanergodiccor2} (c) (i) and (iii) are valid. $\EuScript{S}$ is not mean ergodic since supports of ergodic measures are not uniquely ergodic.
\item If $K$ and $\varphi$ are defined as in Example \ref{PrimExamples} (iii), then supports of ergodic measures are uniquely ergodic and $\EuScript{S}$ is radical free. Thus Theorem \ref{meanergodiccor2} (c) (ii) and (iii) are fulfilled. The semigroup $\EuScript{S}$ is not mean ergodic since the primitive spectrum is not Hausdorff.
\end{enumerate}
\end{example}
\parindent 0pt
\parskip 0.5\baselineskip
\setlength{\footskip}{4ex}
\bibliographystyle{alpha}
\bibliography{literature} 
\footnotesize

  \textsc{Henrik Kreidler, Mathematisches Institut, Universität Tübingen, Auf der Morgenstelle 10, D-72076 Tübingen, Germany}\par\nopagebreak
  \textit{E-mail address}: \texttt{hekr@fa.uni-tuebingen.de}
\end{document}